\documentclass[11]{article}
\usepackage[utf8]{inputenc}
\usepackage{amsmath,amsfonts,amssymb,amsthm}
\usepackage{comment}
\usepackage{hyperref}

\newtheorem{theorem}{Theorem}[section]

\newtheorem{corollary}[theorem]{Corollary}

\newtheorem{definition}[theorem]{Definition}

\newtheorem{lemma}[theorem]{Lemma}

\newtheorem{proposition}[theorem]{Proposition}
\newtheorem{remark}[theorem]{Remark}

\newcommand{\N}{\mathbb{N}} 
\newcommand{\K}{\mathcal{K}}
\newcommand{\D}{{\mathbf D}}
\newcommand{\E}{{\mathbf E}}

\newcommand{\C}{C_{\E}}

\newcommand{\F}{\mathcal{F}}
\newcommand{\G}{\mathcal{G}}

\newcommand{\cR}{\mathcal{R}}
\newcommand{\bD}{\mathbf{D}}
\newcommand{\cS}{\mathcal{S}}
\newcommand{\cP}{\mathcal{P}}
\newcommand{\cQ}{\mathcal{Q}}

\newcommand{\dubCan}{2^{ \N \times \N}}
\newcommand{\lo}{{\bf LO^*}}
\newcommand{\wf}{{\bf WF^*}}

\newcommand{\leex}{{\le ^*_x}}
\newcommand{\lex}{{< ^*_x}}

\begin{document}
\title{Dynamical pair assignments}
\author{Udayan B. Darji
\date{}
\thanks {UBD acknowledges support as a Fulbright Distinguished Scholar of the Czech Fulbright Commission.} \and Felipe Garca-Ramos
\thanks{ FG  was partially supported by  
the Grant U1U/W16/NO/01.03 of the Strategic Excellent Initiative program of
the Jagiellonian University, and the grant K/NCN/000198 of the Narodowe Centrum Nauki (NCN).
}}
\date{\today}
\maketitle

\abstract{
Relations between points in the phase space are central to the study of topological dynamical systems. Since many of these relations share common properties, it is natural to study them within a unified framework. To this end, we introduce the concept of \textit{dynamical pair assignments} $\mathcal{P}$.

We then introduce the notions of a dynamical system being $\mathcal{P}$-full and $\mathcal{P}$-realizable, which generalize several existing concepts in the field like CPE, weak mixing and UPE. Our results establish that the space of $\mathcal{P}$-full systems is always a Borel set, while the space of $\mathcal{P}$-realizable systems is Borel if and only if an associated natural rank is bounded.
%Given a family of pairs one may study the set of all such pairs, naturally introducing a rank of the dynamical systems. This rank is a countable ordinal which measures the complexity of the system. We prove a general theorem which shows that these ranks have regularity properties and fall under what is called co-analytic rank in descriptive theory. As such, one may apply basic descriptive set theory to obtain new examples and properties of dynamical systems in question. 
}
\tableofcontents
\begin{comment}
   TO CHANGE $\Pi_1^1$-ranks are a natural tool for studying coanalytic sets in descriptive set theory. In \cite{Kechris}, Kechris provided a technique to build $\Pi_1^1$-ranks using derivatives. In this note we will prove a variant of this result that is applicable to the $\Gamma$ rank. Some dynamical ranks, like the entropy rank and mean dimension rank can be stated in terms of the $\Gamma$ rank. 
\end{comment}

\section{Introduction}

%{\color{red} Do you want to write this section in light of the abstract? You probably have it in one of your paper in one form or another}

Since the birth of the qualitative theory of dynamical systems, examining dynamical systems from a local perspective has been an integral aspect of the field. Of particular interest to this paper are the various definitions concerning the behavior of a pair of points relative to a dynamical system.

The concept of asymptotic (or homoclinic) pairs can be traced back to the seminal work of Poincaré and has played a significant role in understanding the chaotic properties of Hamiltonian and hyperbolic systems. Regionally proximal pairs, introduced by Veech \cite{veech1968equicontinuous}, serve to characterize the maximal equicontinuous factor of topological dynamical systems. Li-Yorke chaos, a well-known notion of chaos, is defined using scrambled pairs (also known as Li-Yorke pairs). Blanchard established that every dynamical system with positive entropy can be localized using the notion of entropy pairs, leading to the development of the field known as local entropy theory. 

This perspective has proven useful not only for understanding local behavior but also for illuminating connections to other dynamical objects such as the Ellis semigroup, factors, sequence entropy, mean dimension theory, spectrum, nilsystems, and cardinality of ergodic measures, in addition to the aforementioned notions of dynamical pairs. Moreover, it has facilitated novel interactions between dynamical systems and other areas of mathematics, including operator algebras, combinatorics, group theory, descriptive set theory, and continuum theory (see the survey \cite{garcia2024local}).

In this paper, we study from an abstract point of view a formal definition of dynamical pairs. In order to achieve this, we define \textit{ dynamical pair assignment} (Definition \ref{def:dynpair}), which encapsulates the basic properties of several definitions of natural pairs in a general setting.  

Given a dynamical pair assignment $\mathcal{P}$, we introduce the notions of a dynamical system being $\mathcal{P}$-full and $\mathcal{P}$-realizable (see Definition \ref{def:full}), which provide distinct interpretations of having pairs from $\mathcal{P}$ well distributed across the product space $X^2$. Previously established concepts, such as uniform positive entropy, completely positive entropy, weak mixing, and others, fit within these categories under the appropriate setup. 

We also consider a fixed compact metrizable space and study the family of dynamical systems that are $\mathcal{P}$-full and $\mathcal{P}$-realizable. We show that the class of $\mathcal{P}$-full systems is always Borel, which implies that, from a descriptive perspective, membership in this family is relatively straightforward to determine. In contrast, the class of $\mathcal{P}$-realizable systems is generally not Borel, but we provide a rank-based characterization that identifies conditions under which it is Borel.

Each dynamical pair assignment induces a $\Gamma$ rank on a fixed compact metrizable space via topological and transitive closure processes (see Section 3). Several of these hierarchical structures have already appeared in related work \cite{barbieri2020,salo2019entropy,westrick2019topological}. We show that if the rank induced by a dynamical pair assignment $\mathcal{P}$ is unbounded, then the class of $\mathcal{P}$-realizable systems will not be Borel.

To prove the result, we show that the $\Gamma$ rank is a coanalytic rank. Well-known ranks, such as the Cantor-Bendixson rank, can be proven to be coanalytic using a general framework developed by Kechris, which employs the concept of derivatives \cite[Theorem 34.10]{Kechris}. Although the $\Gamma$ rank does not fit precisely within the dual framework of Kechris' result, certain techniques from Kechris' approach can still be applied to demonstrate that the $\Gamma$ rank is indeed a coanalytic rank.

 Some of the results of this paper appeared in an unpublished note by the same authors \cite{darji2021note}.
 
\textbf{Acknowledgment:} The authors thank Dominik Kwietniak,
Slawomir Solecki, and Linda Westrick for motivating conversations. 

The research was finalized during the visit of the U.D. to Jagiellonian University funded by the Strategic Excellent Initiative program at Jagiellonian University

\section{Dynamical pair assignment}
%Throughout this paper $X$ represents a compact metrizable space andh, the diagonal of $X$.

We say $(X,T)$ is a \textbf{topological dynamical system (TDS)} if $X$ is a compact metrizable space and $T\colon X\rightarrow X$ is a continuous function.

When we say that $\varphi\colon (X,T)\to (X_1,T_1)$ is a \textbf{factor map} it means that $(X,T)$ and $(X_1,T_1)$ are TDSs, and $\varphi\colon X\rightarrow X_1$ is a continuous surjection such that $\varphi\circ T_{}=T_{1}\circ\varphi$.
In this situation, we say $(X_1,T_1)$ is a \textbf{factor} of the TDS $(X,T)$. 

 We denote by $C(X,Y)$ the set of all continuous functions from $X$ to $Y$ endowed with the uniform topology and by $K(X)$ the space of compact subsets of $X$ endowed with the Vietoris topology or, equivalently, the Hausdorff metric.

We introduce the following concept. 
\begin{definition}
\label{def:dynpair}
Let $\mathcal{P}$ be a collection of maps of the form \[
\mathcal{P}=\{\mathcal{P}_X|\text{ X is a compact metrizable space and }{\mathcal P}_X\colon C(X,X)\to K(X^2)\}.
\] 
We say that $\mathcal{P}$ is a \textbf{dynamical pair assignment} if the 
 following conditions hold for every $\cP_X\in \cP$:
 
\begin{enumerate}
    \item for every $T\in C(X,X)$, $\cP _X (T)$ is $T\times T$-invariant, 
    \item if $\varphi\colon (X,T)\to (X_1,T_1)$ is a factor map and $(x,y)\in \cP_X(T)$ then $(\varphi(x),\varphi (y))\in \cP_{X_1}(T_1)$, and 
    \item $\cP_X$ is a Borel map. 
\end{enumerate}

\end{definition}
%Throughout we drop the subscript $X$ on $\cP _X$ as the context of the underlying space will be clear whenever we write $\cP (T)$ where $T:X \rightarrow X$ is a continuous map. 

In Section 4 we will see that regionally proximal pairs (Proposition \ref{prop:QisBorel}) as well as entropy pairs (Proposition \ref{prop:EisBorel}) generate dynamical pair assignments. 

%A fundamental question in descriptive set theory is asking whether subsets of a Polish space are Borel or not.  

\textbf{For the rest of the paper $\mathcal{P}$ represents a fixed  dynamical pair assignment.}

\begin{remark}
Note that there is a correspondence between compact metrizable spaces and $\cP_X\in \cP$.  
\end{remark}

Given a compact metrizable space $X$, we use $\Delta_X$ to denote the diagonal of $X$, namely, $\{(x,x): x \in X \}$. 

\begin{lemma}
\label{lem:onepairborel}
Let $\cP_X\in \cP$. The subset
\[
\{T\in C(X,X) : \cP_X(T)\setminus \Delta_X\neq \emptyset\}
 \]
 is Borel.
\end{lemma}
\begin{proof}
    This simply follows from the fact that $K(\Delta_X)$ is a closed subset of $K(X^2)$ and $\cP_X$ is Borel.
\end{proof}

The following definitions are inspired by the concepts of uniform positive entropy, completely positive entropy, and a characterization of topological weak mixing (see Section \ref{sec:examples}). 

\begin{definition}
\label{def:full}
 We say a TDS $(X,T)$ is \textbf{$\mathcal{P}$-full} if $\cP_X(T)=X^2$. We say $(X,T)$ is \textbf{$\cP$-realizable} if the  smallest $T\times T$-invariant closed equivalence relation containing $\mathcal{P}_X(T)$ is $X^2$. 
%{\color{red} is $T\times T$-invariant really necessary? } YES.
\end{definition}

\begin{definition}
\label{def:R_p}
Let $\cP_X\in \cP$. We define the following subsets of $C(X,X)$:
\begin{gather*}
F({\mathcal{P}_X)
}=\{T\in C(X,X): (X,T)\text{ is $\mathcal{P}$-full}\}\text{, and}\\
R(\mathcal{P}_X)=\{T\in C(X,X): (X,T)\text{ is $\mathcal{P}$-realizable}\}.
\end{gather*}

\end{definition}

We are now interested in studying the descriptive complexity of $F(\mathcal{P}_X)
$ and $R(\mathcal{P}_X)$. One of the most basic questions from this point of view is if a set is Borel or not. 

 The following result follows directly from the fact that every $\cP_X\in \cP$ is a Borel function. 
\begin{proposition}
 Let $\cP_X\in \cP$. Then $F(\cP_X)$ is a Borel subset of $C(X,X)$. 
\end{proposition}
\begin{proof}
    This follows from the fact that $\cP_X$ is a Borel map and $F(\mathcal{P}_X)= \cP^{-1}(X \times X)$.
\end{proof}

Studying $R(\mathcal{P}_X)$ is more delicate. We will first establish that it is a coanalytic subset and, in the following section, provide a characterization of the conditions under which it is Borel.

 \begin{lemma}
 \label{lem:factors}
     Let $(X_1,T_1)$ be a nontrivial factor of a $\cP$-realizable TDS $(X,T)$. We have $\cP_{X_1}(T_1)\setminus \Delta_{X_1}\neq \emptyset$. 
 \end{lemma}

\begin{proof} 
%($\Leftarrow)$ To obtain a contradiction, assume that 
    %there exists a closed $T\times T$-invariant equivalence relation $R$ so that $\cP (T)\subset R\neq X\times X$. Let $(X/ R,T_R)$ be the induced TDS by the quotient. We have that $(X/ R,T_R)$ is a non-trivial factor of $(X,T)$. By $\cP (T)\subset R$, and Condition (2) of Definition \ref{def:dynpair}, we have that  $\cP (T_R)\subset \Delta_{X/ R}$, contradicting the hypothesis.($\Rightarrow)$
% Now assume that $(X,T)$ is $\cP$-realizable.
Let $(X,T)$ be a $\cP$-realizable TDS. Assume that we have a factor map $\pi\colon (X,T)\to (X_1,T_1)$ such that $\cP _{X_1} (T_1)\subset \Delta_{X_1}$. Let $R\subset X^2$ be the equivalence relation induced by the factor map $\pi$, that is, $(x,y)\in R$ if and only if $\pi(x)=\pi(y)$. One can check that $R$ is necessarily a closed $T\times T$-invariant equivalence relation.
 Furthermore, for every $(x,y)\in \cP_X(T)$, Definition \ref{def:dynpair} (2) implies $(\pi(x),\pi(y))\in \cP_{X_1}(T_1)$. As $\cP_{X_1} (T_1)\subset \Delta_{X_1}$, we get $\pi(x)=\pi(y)$, implying that $(x,y) \in R$. Thus, $\cP_X(T) \subset R$. As $(X,T)$ is $\cP$-realizable, then $R=X^2$ and $X_1$ are trivial.   
\end{proof}

%A subset of a Polish space is \textbf{analytic} if it is the continuous projection of a Borel subset. A subset of a Polish space is \textbf{coanalytic} if it is the complement of an analytic subset.
Recall that a subset of a Polish space is \textbf{analytic} (or $\Sigma^1_1$) if it is the continuous image of a Borel set of a Polish space. Complements of the analytics subsets are called \textbf{coanalytic} (or $\Pi^1_1$). A subset is Borel if and only if it is analytic and coanalytic. 

\begin{proposition}
\label{prop:R_p}
    Let $\cP_X\in \cP$. We have that $R(\cP_X)$ is a coanalytic subset of $C(X,X)$. 
\end{proposition}
\begin{proof}
Let $I^{\N}$ be the Hilbert cube and consider
\begin{equation*}
\begin{split}
  {\mathcal A} = & \{(T, \varphi, S): T \in C(X,X), \ \varphi \in C(X,I^{\N}), \ |\varphi (T)| >1, \ \\&  S \in C  (\varphi(T), \varphi(T) ), \  \varphi \circ T = S \circ \varphi  \textit { and } \cP(S) =\Delta_{\varphi(T)}
 \}.  
\end{split}
\end{equation*}

It is easy to see that all except the last condition in the definition of ${\mathcal A}$ give rise to a Borel subset. To show that the last condition ($\cP(S) =\Delta_{\varphi(T)}$) also gives a Borel subset, use Lemma \ref{lem:onepairborel}; thus ${\mathcal A}$ is Borel.
%{\color{red} in which space??}. 
Hence, $\Pi_1({\mathcal A})$,  its projection onto the first coordinate, is analytic. By Lemma \ref{lem:factors}, $\Pi_1({\mathcal A})$  is precisely the set of all $T \in C(X,X)$
 which do not belong to $R(\cP _X)$. Hence, $R(\cP _X)$, the complement of $\Pi_1({\mathcal A})$, is coanalytic.
\end{proof}

%In the following section we will characterize when $R_{\mathcal{P}}(X)$ is a Borel subset. 

\section{$\Gamma$ rank and $\cP$ rank}
\label{sec:gamma}

As usual in descriptive set theory we will use $\N$ to denote the non-negative integers, and make a correspondence between ordinals and subsets of ordinals (e.g. $\N=\omega$). 

To prove that a TDS $(X,T)$ is $\cP$-realizable, one needs to be able to construct the smallest closed equivalence relation that contains $\cP_X(T)$. This set can be constructed with transfinite induction using the $\Gamma$ rank defined as follows. 

Let $Z$ be a set, and $F\subseteq Z^2$. Define the \textbf{transitive closure} of $F$ as
\[
F^{+}=\{(u,v)\in Z^2:\exists u=u_1,...,u_n=v \text{ s.t. } (u_i,u_{i+1})\in F\}.
\] 
A set $F\subseteq Z^{2}$ is an equivalence relation if and only if $F$ is symmetric,
$\Delta_{Z}\subseteq F$, and $F=F^{+}$. In this case, denote the equivalence
class of $z\in Z$ by $[z]_{F}$.

Let $E\subseteq X^{2}$ be a symmetric set. Define
\[
\Gamma(E)=\overline{E^{+}\cup \Delta_X}.
\]
For a countable ordinal $\alpha$, $\Gamma^{\alpha}(E)$ is defined by 
\[
\Gamma^{\alpha}(E) = \Gamma( \Gamma^{\alpha -1}(E )) 
\]
if $\alpha$ is the successor ordinal and 
\[
\Gamma^{\alpha}(E) = \overline {\cup_{\beta < \alpha} \Gamma^{\beta}(E)}
\]
if $\alpha$ is a limit ordinal.  

Recall that in a compact metrizable space, a chain of strictly increasing sequence of closed sets must be countable. Therefore, the following proposition holds.
\begin{proposition}
Let $X$ be a compact metrizable space.
There exists a countable ordinal $\alpha$ such that $\Gamma^{\alpha}%
(E)=\Gamma^{\alpha+1}(E)$.
%do we need symmetric?
\end{proposition}

 The smallest such ordinal is called the $\Gamma$ \textbf{rank of} $E$. Define $\Gamma^{\infty} (E)=\Gamma^{\alpha} (E)$, where $\alpha$ is the rank of $E$. 

\begin{definition}
    Let $\cP_X\in \cP$ and $T\in C(X,X)$. We define the $\cP$\textbf{ rank} of $T$ as the $\Gamma$ rank of $\cP_X(T)$. 
    
    We say that the $\cP$ rank is \textbf{bounded on} $X$, if there exists a countable ordinal $\alpha$ such that for every $T\in C(X,X)$, the $\cP$ rank of $T$ is bounded by $\alpha$. 
\end{definition}

In Section 5 we will prove that the $\cP$ rank is a coanalytic rank (Theorem \ref{mainpi11}) and obtain the following consequence.

\begin{theorem}\label{thm:GammaBorel}
    We have that $R(\cP_X)$ is a Borel subset of $C(X,X)$ if and only if the $\cP$ rank is bounded on $X$. 
\end{theorem}

\section{Examples}
 \label{sec:examples}

\subsection{Entropy pairs}

%\subsection{Entropy rank}

Let $(X,T)$ be a TDS and $\mathcal{U},\mathcal{V}$ open covers of $X$. We denote the smallest cardinality of a subcover of $\mathcal{U}$ with $N(\mathcal{U})$, and we define
\[
\mathcal{U}\vee\mathcal{V}=\{U\cap V:U\in\mathcal{U}\text{ and }V\in\mathcal{V}\}.
\] We define the \textbf{entropy of
$(X,T)$ with respect to $\mathcal{U}$} as
\[
h_{\text{top}}(X,T,\mathcal{U})=\lim_{n\rightarrow\infty}\frac{1}{n }\log N(\vee^n_{m=1}T^{-m}(\mathcal{U})).
\]
The \textbf{(topological) entropy} of $(X,T)$ is defined as \[
h_{\text{top}}(X,T)=\sup_{\mathcal{U}}h_{\text{top}}(X, T,\mathcal{U}).
\]

We will now define common concepts from local entropy theory. 
The local entropy theory was initiated in \cite{blanchard93}. For more information, see the surveys \cite{garcia2024local,glasner2009local}. 

\begin{definition}
A TDS has \textbf{complete positive entropy (CPE)}  if every non-trivial factor has positive entropy. 

A TDS $(X,T)$ has \textbf{uniform positive entropy (CPE)} if for every open cover $\mathcal{U}$ composed of two non-dense open sets we have that \[h_{\text{top}}(X,T,\mathcal{U})>0.\]
\end{definition}

%In the next section we will see how to characterize CPE using independence entropy pairs. 

 A set $I \subseteq \N$ has \textbf{positive density} if
$\liminf_n \frac{|I \cap [1,n]|}{n} >0$.

 Given a TDS $(X,T)$ and $\{U,V\}\subset X$, we say $I \subset \N$ is an \textbf{independence set for $\{U,V\}$} if for all finite $J \subseteq I$,
and for all $(Y_j) \in \prod _{j\in J}{\mathcal A}$, we have that
$$\cap_{j\in J}T^{-j}(Y_j)\neq \emptyset.$$

\begin{definition}
Let $(X,T)$ be a TDS. We say that $(x_1, x_2) \in X\times X$ is an \textbf{independence entropy pair (IE-pair) of $(X,T)$} if for every pair of open sets $A_1,A_2$, with $x_1\in A_1$ and $x_2\in A_2$, there exists an independence set for $\{A_1,
A_2\}$ with positive density. The set of IE-pairs of $(X,T)$ will be denoted by $E(X,T)$.
\end{definition}

For the proof of the following results, see \cite[Theorem 12.19]{kerr2016ergodic} and \cite[Proposition 3.9 and
Theorem 3.15]{KerrLiMA}. Also see \cite{huang2007relative}. 
\begin{theorem}
  %[\cite{huang2007relative,KerrLiMA}]    
  \label{thm:epair}
Let $(X,T)$ be a TDS.

\begin{enumerate}
    \item $E(X,T)$ is a closed $T\times T$-invariant set. 
    \item $E(X,T)\setminus \Delta_X\neq \emptyset$ if and only if $h_{\text{top}}(X,T)>0$. 

    \item If $\varphi\colon (X,T)\to (X_1,T_1)$ is a factor map and $(x,y)\in E(X,T)$ then $(\varphi(x),\varphi (y))\in E(X_1,T_1)$.
\end{enumerate}

\end{theorem}

\begin{definition}
   We define the \textbf{entropy pair assignment} as the collection of maps
\begin{align*}
\mathcal{E} = \left\{ \mathcal{E}_X \mid \text{X is a compact metric space, } \right. & \\
\left. {\mathcal{E}}_X \colon C(X,X) \to K(X^2) \text{, and } \mathcal{E}_X(T) = E(X,T) \right\}.
\end{align*}
\end{definition}

  % \[\mathcal{E}=\{\mathcal{E}_X|\text{ X is a compact metric space, }{\mathcal P}_X\colon C(X,X)\to K(X^2)\text{, and }\mathcal{E}_X(T)=E(X,T)\}.\] 
 %  $\mathcal{E}_X\colon C(X,X)\to K(X^2)$ given by $\mathcal{E}_X(T)=E(X,T)$ for any compact metrizable space $X$.

\begin{proposition}\label{prop:EisBorel}
    We have that $\mathcal{E}$ is a dynamical pair assignment.
    \end{proposition}
\begin{proof}  
     Considering Theorem \ref{thm:epair}, it remains to show that $\mathcal{E}_X$ is a Borel map for compact metric space $X$.
    Let $U,V$ be open in $X$. We first observe that
\[\{ T \in C(X,X):  \mathcal{E}_X(T) \cap (U \times V)\neq \emptyset\}\tag{\(\dagger\)}
\]
is Borel. Indeed, using an equivalent condition of independence given in \cite[Lemma 3.2]{KerrLiMA}, we have that $\dagger$ is satisfied by $T$ if and only if there is a rational number $ r >0 $ such that for all $l \in \N$ there is an interval $I \subseteq \N$ with $|I| \ge l$ and a finite set $F \subseteq I$ with $|F| \ge r |I|$ such that $F$ is an independent set for $(U,V)$.
It is easy to verify that for fixed $U,V, r, l, I, F$ set
\[
   \{T \in C(X,X): F \text{ is an independent set for } (U,V) \text{ for } T  \tag{\(\ddagger\)} \}  
\]
is open. Now, the set in $\dagger$ is the result of a sequence of countable unions and countable intersections of sets of type $\ddagger$. Hence, $\dagger$ is Borel. Since $X$ has a countable basis, taking unions, we find that $\dagger$ is Borel when $U \times V$ is replaced by any open set $W \subseteq X \times X$. Every closed set in $X \times X$ is the monotonic intersection of a sequence of open sets in $X \times X$. This and the fact that $E(X,T)$ is closed imply that $\dagger $ is Borel when $U \times V$ is replaced by a closed set $C \subseteq X \times X$. Reformulating the last statement, we have that for all open $W \in X \times X$, the set 
\[\{ T \in C(X,X):  \ E(X,T) \subseteq W \}\tag{\(\diamond\)}
\]
is Borel. Putting $\dagger$ and $\diamond$ together, we have that 
\[\{T \in C(X,X):  E(X,T) \subseteq \cup_{i=1}^n  (U_i \times V_i) \ \  \&  \  E(X,T) \cap (U_i \times V_i) \neq \emptyset, 1 \le i \le n \}
\]
is Borel whenever $U_1, \ldots,U_n, V_1, \ldots V_n$ are open in $X$, completing proof.
\end{proof}

For a proof of the following result, see \cite[Subsection 2.4]{garcia2024local} and the references therein. 
\begin{proposition}
\label{prop:CPEUPE}
We have that
\[
R_{\mathcal{E}}(X)=\{T\in C(X,X): (X,T)\text{ has CPE}\}\text{, and}
\]
\[
F_{\mathcal{E}}(X)=\{T\in C(X,X): (X,T)\text{ has UPE}\}.
\]

\end{proposition}

Using Proposition \ref{prop:EisBorel}, Proposition \ref{prop:CPEUPE} and Theorem \ref{thm:GammaBorel} we conclude the following result. 
\begin{corollary}
The subset $R_{\mathcal{E}}(X)$ is Borel if and only if the $\mathcal{E}$ rank is bounded on $X$.
\end{corollary}

Examples of TDSs with CPE and arbitrarily high $\mathcal{E}$ rank have been constructed in \cite{barbieri2020,salo2019entropy,darjiGR}. In particular the previous result is used in \cite{darjiGR} to prove that the for Cantor spaces $X$ the intersection of $\text{CPE}(X)$ and topological mixing maps is not Borel. 

\begin{remark}
  In \cite{KerrLiMA}, other independence pairs were defined, IN-pairs and IT-pairs. These pairs characterize positive sequence entropy and tame systems. The construction of dynamical pair assignments obtained from these notions should be analogous. 
\end{remark}

\subsection{Regionally proximal pairs}
Regionally proximal pairs were defined by Veech \cite{veech1968equicontinuous}, and have been extensively studied (see \cite{auslander1988minimal}). 
\begin{definition}
    Let $(X,T)$ be a TDS. We say $(x,y)\in X^2$ is a \textbf{regionally proximal pair of $(X,T)$} if for every $\varepsilon>0$ there exist $x',y'\in X$ and $n\in \N$ such that $d(x,x'),d(y,y')\leq \varepsilon$ and $d(T^nx',T^ny')\leq \varepsilon$. The regionally proximal pairs of the TDS $(X,T)$ will be denoted by $Q(X,T)$.
\end{definition}

It is not difficult to check that $Q(X,T)$ is $T\times T$-invariant and closed. However, it is not necessarily an equivalent relation. 

\begin{definition}
   We define the \textbf{regionally proximal pair assignment} as the collection
      \begin{align*}
\mathcal{Q} = & \left\{ \mathcal{Q}_X \mid \text{X is a compact metrizable space,} \right. \\
 &{\mathcal P}_X \colon C(X,X) \to K(X^2) \text{, and } \mathcal{Q}_X(T) = Q(X,T)\}.
\end{align*}
\end{definition}

\begin{proposition}
\label{prop:QisBorel}
    We have that $\mathcal{Q}$ is a dynamical pair assignment.
    \end{proposition}
\begin{proof}
Let $X$ be a compact metrizable space. 
It is well known and easy to check that $Q_X(T)$ is $T\times T$ invariant and closed. Hence, (1) of Definition~\ref{def:dynpair} is satisfied.
Property (2) of Definition~\ref{def:dynpair} is also well known, but we will write the proof for the sake of completeness. Let $\varphi\colon (X,T)\to (X_1,T_1)$ be the factor map, $(x,y)\in \cQ_X(T)$ and $\varepsilon >0$. Let $\delta$ correspond to the uniform continuity of $\varphi$ with respect to $\varepsilon$. Let $x',y'\in X$ and $n\in \N$ be such that $d(x,x'),d(y,y')\leq \delta$ and $d(T^nx',T^ny')\leq \delta$. Then, 
\[d(\varphi(x),\varphi(x')),d(\varphi(y),\varphi(y'))\leq \varepsilon\text{ and}\]
\[d(T_1^n(\varphi( x')), T_1^n(\varphi (y'))=d(\varphi(T^n (x')), \varphi(T^ny'))\leq \varepsilon,\] verifying that $(\varphi(x), \varphi(y))$ is a regionally proximal pair of $(X_1,T_1)$. Hence, we have shown that $\varphi( \cQ_X(T)) \subseteq \cQ_{X_1}(T_1)$. 

%By the continuity of $\varphi$, we have that $\varphi(\overline{ \cQ_X(T)}) \subseteq \overline{\cQ_{X_1} (T_1)}$, i.e., $\varphi (\cQ_{X} (T)) \subseteq \cQ_{X_1} (T_1)$. 

Next, we show that $\cQ_X$ is Borel. First observe that replacing $\le$ by $<$ leaves the definition of regionally proximal pair unchanged. We proceed as in the proof of Proposition~\ref{prop:EisBorel}.  For any set $U, V \subseteq X$,
\[A (U,V): =\{ T \in C(X,X):  \mathcal{Q}_X(T) \cap (U \times V )\neq \emptyset\}. \]
It will suffice to show that $A(U,V)$
is Borel whenever $U,V$ are open in $X$. The rest of the proof follows as the proof of Proposition~\ref{prop:EisBorel}. 

For any open sets $U,V$ in $X$, and $m \in \N$, we have that 
\begin{multline*}
    A(m, U,V) := 
  \{T \in C(X,X): \exists (x,y) \in U \times V, \\
 \ n \in \N \mbox{ such that }d(T^n(x),T^n(y)) < 1/{m}\}
\end{multline*}
is an open subset of $C(X,X)$. Hence, 
$\cap_{m=1}^{\infty} A(m, U,V)$ is $G_{\delta}$ (countable intersection of open subsets). Moreover, as $X$ compact, we have that  $\cap_{m=1}^{\infty} A(m, U,V) \subseteq A(\overline{U}, \overline{V}).$

Now fix open sets $U,V$ in $X$, and let $\{U_n\}$ and $\{V_n\}$ be increasing sequences of open sets such that $\overline {U}_n \subseteq U$, $\overline {
V}_n \subseteq V$ and $\cup U_n = U$ and $\cup V_n = V$.
Then, we have that 
\[ A (U,V)= \cup_{n=1} ^{\infty}\cap_{m=1}^{\infty} A(m, U_n,V_n),\]
verifying that $A(U,V)$ is $G_{\delta\sigma}$, in particular Borel, and completing the proof.   
\end{proof}

A TDS is \textbf{equicontinuous} if for every $\varepsilon>0$ there exists $\delta>0$ such that $d(T^nx,T^ny)\leq \varepsilon, \ \forall n\in \N$ whenever $d(x,y)\leq \delta$.

It is well known that a TDS $(X,T)$ is equicontinuous if and only if $\mathcal{Q}_X(T)=\Delta_X$ (e.g. see \cite[Page 126]{auslander1988minimal}). This implies that equicontinuity plays a similar role for regional proximal pairs, as zero entropy to entropy pairs. 
\begin{proposition}
 If a TDS $(X,T)$ is $\mathcal{Q}$-realizable, then every non-trivial factor is not equicontinuous.   
\end{proposition}
  \begin{proof}
     % Assume that $(X,T)$ is not $\mathcal{Q}_X$-realizable. Let  $\mathcal{R}$ be the smallest closed $T\times T$-invariant equivalence relation that contains $\mathcal{Q}$. By the assumption $\mathcal{R}\neq X^2$. Let $f_{\mathcal{R}}\colon X\to X/\mathcal{R}$ be the quotient map. There is a natural mapping $T_{\mathcal{R}}:X/\mathcal{R}$ that makes $f\colon (X,T)\to (X_{\mathcal{R}},T_{\mathcal{R}}$ a factor map. 
      Assume there is a non-trivial factor map $\varphi\colon (X,T)\to (X_1,T_1)$ such that $(X_1,T_1)$ is equicontinuous.  Let $\mathcal{R}$ be the relation on $X$ defined by $(x,y) \in {\mathcal R}$ if and only if $(\varphi(x), \varphi(y)) \in Q_{X_1}(T_1)$. As $Q_{X_1}(T_1)=\Delta_{X_1}$, we have that $\mathcal {R}$ is a closed $T$-invariant equivalence relation. 
       Using (2) of Definition \ref{def:dynpair} we obtain that $\mathcal{Q}_X(T)\subset \mathcal{
      R
      }$. Since $(X_1,T_1)$ is non-trivial, there exists $x,y \in X$ such that $\varphi(x) \neq \varphi(y)$, implying that $(x,y) \notin {\mathcal R}$, i.e., $(X,T)$ is not $\mathcal{Q}$-realizable. 
  \end{proof}

Regionally proximal pairs are particularly well studied for minimal TDSs. We say that a TDS $(X,T)$ is $\textbf{minimal}$ if for every non-empty closed subset $K\subset X$ with $T(K)\subset K$ we have $K=X$.

If $(X,T)$ is a minimal TDS then $\mathcal{Q}_X$ is an equivalence relation \cite[Chapter 9, Theorem 8]{auslander1988minimal}. This implies the following.
\begin{proposition}
A minimal TDS is $\mathcal{Q}$-full if and only it is $\mathcal{Q}$-realizable.     
\end{proposition}
%\begin{proof}
    
%\end{proof}

\begin{remark}
    Other dynamical pairs that we will not study in this paper but we suspect that they fit in the framework of dynamical pair assignments are proximal pairs \cite{auslander1988minimal}, asymptotic pairs (e.g., see \cite{barbieri2020}), transitivity pairs \cite{westrick2019topological} and mean dimension pairs \cite{garciagutman}.
\end{remark}

\section{Coanalytic ranks}
Coanalytic ranks provide a natural framework to determine when a coanalytic subset is Borel. We refer the reader to \cite[Chapter 34]{Kechris} as a general reference to the topic. We will prove that the $\cP$ rank and the $\Gamma$ rank are coanalytic (Corollary \ref{cor:gamma} and Theorem \ref{GammaBorel}). As a consequence, we will prove Theorem \ref{thm:GammaBorel}.

%Given a Polish space $X$, and a subset $A\subset X$, we denote the complement of $A$ by $A^c$

\begin{definition}
Let $C$ be a set. A {\bf rank} on $C$ is simply a function $\varphi:C\rightarrow\omega_1$, where $\omega_1$ is the set of countable ordinals. 
\end{definition}
%Given a rank on $C$, $\varphi$, we obtain the order relations $<_{\varphi}$ and $\le_{\varphi}$ defined as follows:\[ x <_{\varphi} y \iff \varphi(x) < \varphi (y)\]\[ x \le_{\varphi} y \iff \varphi(x) \le \varphi (y).\]
%When $P$ is a subset of a Polish space $X$, we extend $\varphi$ to all of $X$ by letting $\varphi(x) = \omega_1$ for $x \notin P$, ($\omega_1$ is the first uncountable ordinal). Hence we extend $<_{\varphi}$ and $\le_{\varphi}$ by 
%\[ x <_{\varphi}^* y \iff x \in P \ \ \& \ \  \varphi(x) < \varphi (y)\]
%\[ x \le_{\varphi}^* y \iff x \in P \ \ \& \ \  \varphi(x) \le \varphi (y)\]
\begin{definition}\label{pi11}\cite[Section 34B]{Kechris}
Let $X$ be a Polish space, $C \subseteq X$ a coanalytic subset, and $\varphi : C \rightarrow \omega_1$ a rank on $C$. We say that $\varphi$ is a  {\bf  coanalytic rank} if there are relations $P, Q \subseteq X^2$, with $P$ and $X^2\setminus Q$ analytic, such that for all $y \in C$ we have that 
\begin{gather*}
    \{x \in C: \varphi (x) \le \varphi(y)\} = \{x\in X: (x,y) \in P\} = \{x \in X: (x,y) \in Q\}.
\end{gather*}
%are coanalytic subsets of $X \times X$.
%\[ x <_{\varphi}^* y \ \ \ \ \ \  x \le_{\varphi}^* y\]
\end{definition}
Loosely speaking, $\varphi$ is a coanalytic rank if $\{x: \varphi (x) \le \varphi(y)\} $ is "uniformly Borel in $y$".

A classical example of a coanalytic rank is the Cantor-Bendixson rank. In fact, every coanalytic set admits a coanalytic rank \cite{Kechris}.

We will use the following reformulation of coanalytic rank.
\begin{proposition}\label{altpi11}\cite[Exercise 34.3]{Kechris} Let $X$ be a Polish space, $C\subset X$ a coanalytic subset and $\varphi$ a rank. Then, $\varphi$ is a coanalytic rank if and only if there  are analytic relations $P, Q \subseteq X^2$ such that for all $y \in C$ we have that
    \begin{align*}
    & \{x \in C: \varphi (x) \le \varphi(y)\} = \{x\in X: (x,y) \in P\},\text{ and} \\ 
    & \{x \in C: \varphi (x) < \varphi(y)\}= \{x \in X: (x,y) \in Q\}.
\end{align*}
\end{proposition}

%\begin{theorem}Let $C$ be a coanalytic set and $\varphi$ be a coanalytic rank on $C$. If $A \subseteq C$ is analytic, then $\varphi$ is bounded on $A$, i.e., there exists $\alpha < \omega_1$ such that $\varphi(x) < \alpha$ for all $x \in A$. In particular,\[ C \mbox{ is Borel} \iff \varphi \mbox{ is bounded on } C.\]\end{theorem}

\subsection{Derivatives}
In this subsection, we present examples of coanalytic ranks. Although this section does not directly contribute to the main results of the paper, it serves to motivate the definitions and provide context for the results.

We recall the notion of derivatives from \cite[Section~34.D]{Kechris}.
%Let $\K(X)$ denote the space of all compact subsets of $X$ endowed with the Hausdorff metric. 
\begin{definition} 
A map $\D:K(X) \rightarrow K(X)$ is a {\bf derivative} if the following holds:
\[ \D(A) \subseteq A \ \ \ \ \& \ \ \ A \subseteq B \implies \D(A) \subseteq \D(B). \] 

%A map $\E:\K(X) \rightarrow \K (X)$ is an {\bf expansion}
%if the following holds:
%\[ A \subseteq B \iff \E(A) \supset \E(B). \]
\end{definition}
 Derivatives appear in a variety of contexts and Borel derivatives induce coanalytic ranks in a natural way. For  a derivative $\D$, let 
 \begin{align*}
    &\D ^{0} (A)  =  A\\
    &\D ^{\alpha+1}  =  \D (\D^{\alpha}(A))\\
    &\D ^{\lambda} (A) = \cap _{\beta < \lambda} \D^{\beta}(A) \textit { if } \lambda \textit{ is a limit ordinal.}
\end{align*}
Let $A \in K(X)$. Then, there exists a countable ordinal $\alpha$ such that $\D^{\alpha} = \D^{\alpha+1}$. Such an ordinal exists since in a separable metric space a chain of strictly decreasing sequence of closed sets must be countable. We let $|A|_{\D}$ be the least such $\alpha$. Moreover, we let $\D^{\infty} (A)=\D^{|A|_{\D}}$, that is, the stable part of $A$. 

A classical Borel derivative is the Cantor-Bendixson derivative given by 
\[ A \rightarrow A' \] 
where $A'$ is the set of limit-points of $A$ \cite[Theorem 6.11]{Kechris}. The $\alpha^{th}$ Cantor-Bendixson derivative of $A$ is denoted by $A^{\alpha}$, $|A|_{CB}$ denotes least ordinal $\alpha$ such that $A^{\alpha +1} = A^{\alpha}$, and $A^{\infty} = A^{|A|_{CB}}$, that is, the stable part of $A$.

The following is an important theorem that relates derivatives to coanalytic ranks.
\begin{theorem}\label{thm:BDerivatives} \cite[Theorem 34.10]{Kechris}
Let $\D:K(X) \rightarrow K(X)$ be a Borel derivative and 
\[ C = \{A \in \K(X): D^{\infty}(A) = \emptyset\}.\]
Then, $C$ is coanalytic and $\varphi : C \rightarrow \omega_1$ defined by $\varphi(A)= |A|_{\D}$ is a coanalytic rank on $C$.
\end{theorem}

\subsection{Expansions}
In \cite[p. 270]{Kechris}, the concept of an expansion operator $E$ is introduced as the complement of a derivative operator $\bD$. Specifically, for open sets $A$, the expansion is defined by $E(A) = X \setminus \bD (X \setminus A)$. However, this notion is not directly applicable for our situation, as we require expansions suitable for the space $K(X)$, which we define below.

\begin{definition}
Let $X$ be a Polish space. We say that $\E$ is an expansion if $\E:K(X) \rightarrow K(X)$ and \[A \subseteq \E(A) \ \ \ \ \& \ \ \  A \subseteq B \implies \E(A) \subseteq \E(B).\]
\end{definition}
Let $X$ be a Polish space, $E$ an expansion, and $\alpha,\lambda$ countable ordinals, with $\lambda$ a limit ordinal. We have the following notation.
\begin{align*}
    &\E ^{0} (A) =  A\\
    &\E ^{\alpha+1} (A) =  \E (\E^{\alpha}(A))\\
    &\E ^{\lambda} (A) = \overline{\cup _{\beta < \lambda} \E^{\beta}(A)}.
\end{align*}
We denote the smallest countable ordinal, $\alpha$, such that $\E ^{\alpha+1}(A)  =  \E^{\alpha}(A)$ by $|A|_{\E}$. We have that $|A|_{\E}$ is a countable ordinal because $X$ is a separable metrizable space. 
Given $J\subset K(X)$, the $\E$\textbf{ rank} on $J$ refers to the function
\[
\varphi_{\E}\colon J\to \omega_1
\]
defined by $\varphi(A)= |A|_\E$ for $A\in J$.
We set $\E ^{\infty} (A)=\E ^{|A|_{\E}}$.
Finally, we represent the exhaustive subsets as
\[ \C= \{A \in K(X): \E ^{\infty} (A) =X \}.
\]

%This leads to the following analogue of Theorem~~\ref{thm:BDerivatives}.

\begin{theorem}\label{BorelExp}
Let $X$ be a compact metrizable space and $\E$ a Borel expansion.
Then $\C$ is a coanalytic subset and the $\E$ rank, $\varphi_{\E}$, on $\C$ is a coanalytic rank.
%and $\varphi : \C \rightarrow \omega_1$, defined by $\varphi(A)= |A|_\E$, is a coanalytic rank on $\C$.
\end{theorem}

Before proving the theorem, let us show a specific instance of a Borel expansion, the map $\Gamma$.

%\begin{definition}Let $X$ be a compact metric space and $E\subseteq X^{2}$ a closed set.We define $E^+$ as the smallest equivalence relation that contains$E$ and $\Gamma(E)=\overline{E^+}$. For an ordinal $\alpha$, $\Gamma^{\alpha}(E)$ is defined by\[\Gamma^{\alpha}(E) = \Gamma( \Gamma^{\alpha -1}(E )) \]if $\alpha$ is the successor ordinal and\[\Gamma^{\alpha}(E) = \overline {\cup_{\beta < \alpha} \Gamma^{\beta}(E)}\]if $\alpha$ is a limit ordinal.  \end{definition}

%Again from the separability of the space, we have that there exists a countable ordinal $\alpha$ such that $\Gamma^{\alpha}(E)=\Gamma^{\alpha+1}(E)$.The smallest ordinal that satisfies the statement in the previous proposition is called the $\Gamma$-\textbf{rank of} $E$.

Let $X$ be a compact metrizable space and $\Gamma\colon K(X\times X) \to K (X\times X)$ the map defined in Section \ref{sec:gamma} ($\Gamma(E)=\overline{E^+}$). It is clear that $\Gamma$ is an expansion. 

Before proving that $\Gamma$ is a Borel map, we present the following basic lemma. 
\begin{lemma}\label{BorelLimit}
Let $X$ be a compact metrizable space, and $\varphi_n: K(X) \rightarrow K(X)$ a Borel function for every $n \in \N$. The map $\varphi:K(X) \rightarrow K(X)$, defined by
\[ \varphi(A) := \overline{\cup_{n=1}^{\infty} \varphi_n (A)},
\]
 is Borel.
\end{lemma}
\begin{proof}
Define $\psi _n :K(X) \rightarrow K(X)^n$ by
\[ \psi _n (A) := (\varphi_1(A), \ldots, \varphi_n(A)).
\]
Then, $\psi_n$ is Borel. Moreover, as the union map is continuous, we have that, for each $n \in \N$, $A \rightarrow \cup_{i=1}^n  \varphi_i(A)$ is Borel. Now $\varphi$ is simply the pointwise limit of these maps and hence itself Borel. 
\end{proof}
\begin{proposition}\label{GammaBorel}
Let $X$ be a compact metrizable space. Then $\Gamma$
is a Borel expansion on $K(X\times X)$. 
\end{proposition}
\begin{proof}
That $\Gamma$ is an expansion is clear from the definition. Let us show that $\Gamma$ is Borel. 
Define $\sim_n: K (X \times X) \rightarrow K (X \times X)$ by 
\[\sim_n (A) :=\{(x,y):\exists x=x_0, \ldots x_n =y \mbox { such that } (x_i, x_{i+1}) \in A \ \  \forall \ 0 \le i <n\}.\] We note that $\sim_n$ is a continuous map. Using Lemma~\ref{BorelLimit}, we conclude that 
\[ \Gamma(A) = \overline{ \cup_{n=1}^{\infty} \sim_n (A)}
\]
is Borel.
\end{proof}
We now proceed to prove Theorem~\ref{BorelExp}. We follow a similar outline of the proof of \cite[Theorem~34.10]{Kechris} adapted to our situation, which involves using a different combinatorial model for coanalytic sets. 

\begin{proof}[Proof of Theorem~\ref{BorelExp}]

We first show that $\C$ is coanalytic. As $\E$ is Borel, $gr(\E)$, the graph of $E$, is also Borel. Let $\Delta$ be the diagonal of \[
\Delta=\{(A,A): A\in K(X) \setminus \{X\}\}.\] Then, $gr(\E)\cap \Delta$ is Borel. As
\[\F = \{A \in K(X): \E(A)=A  \ \&\  A \neq X\}\] is the bijective projection of the Borel set $gr(\E)\cap \Delta$, we have that $\F$ is Borel. 
Let
\[ \G = \{(A,B) \in K(X) \times K(X): B \in \F \ \ \& \ \ A   \subseteq B\}.
\]
We have that
\[
\G=\{(A,B) \in K(X) \times K(X): A   \subseteq B\} \cap K(X) \times \F.
\]
Thus $\G$ is Borel ( intersection of two Borel subsets). 
As the projection of Borel sets are analytic, and  \[A \notin \C \Longleftrightarrow  (A,B) \in \G \textit{ for some B},
\]
we have that $K(X) \setminus \C$ is analytic and consequently $\C$ is coanalytic.

Now we will show that $\varphi_{\E}|_{\C \setminus \{X\}}$ is a coanalytic rank. This is enough to conclude the result. 

We proceed to construct the analytic sets $P,Q$ as in Proposition~\ref{altpi11}. 
%In order to do this we need a  a variant of a standard combinatorial  coanalytic set as in the proof of \cite[Theorem~34.10]{Kechris}. We recall the basic terminology.

For $x \in \dubCan$ we let \[D^*(x) = \{m \in \N: x(m,m) =1\}\] and define $\leex$ as the relation on $D^*(x)\times D^*(x)$ that satisfies \[m \leex n \Leftrightarrow [m, n \in D^*(x) \ \&  \ x(m,n) =1]. \]
By $m <_x^* n$, we mean that $m \leex n$ and $m \neq n$. 
We let $\lo$ be the set of all $x \in \dubCan$ such that $\leex$ is a linear order, $0 \in D^*(x)$ and $0 \leex m$ for all $m \in D^*(x)$ and let $\wf$ be the set of all $x \in \lo$ such that $\leex$ is a well-ordering. 
For $x \in \wf$, we denote the order type of %$x$
$\leex$ by $|x|^*$.
From \cite[p. 273]{Kechris} we obtain the following facts.
\begin{enumerate}
    \item $\lo$ is a closed subset of $\dubCan$,
    \item $\wf$ is a %complete 
    coanalytic subset of $\dubCan$ and $x \mapsto |x|^*$ is a coanalytic rank on $\wf$, and
    \item $\{|x|^*: x \in \wf\}=\omega_1 \setminus \{0\}$.
\end{enumerate}
%It is known that $\lo$ is a closed subset of $\dubCan$ and $\wf$ is a complete coanalytic subset of $\dubCan$ and $x \mapsto |x|^*$ is a coanalytic rank on $\wf$ where $|x|^*$ is the order type of $x \in \wf$. Moreover, the range of $\wf$ is $\omega_1 \setminus \{0\}$. \cite[p. 273]{Kechris}

Assume that there exist analytic subsets $\cR,\cS\subset\lo \times K(X)$ which satisfy the following properties (we will prove the existence of these sets later). 
\begin{align*}
    & \{x \in \lo : (x, A) \in \mathcal{R} \} = \{x \in \wf : |x|^* \leq |A|_{\E}\} \quad \forall A \in \mathcal{C} \setminus \{X\} \tag{R} \\
    & \{A \in K(X) : (x, A) \in \mathcal{S} \} = \{A \in \C : |x|^* = |A|_{\E}\} \quad \forall x \in \wf. \tag{S}
\end{align*}

Using the sets above, we define sets $P$ and $Q$ and verify that they satisfy the desired properties.
Let 
\[ P = \{(A,B)  \in K(X) ^2 : \exists x \in \lo \textit{ such that } (x,B) \in \cR \ \& \ (x,A) \in \cS\}. \]
Then, $P $ is analytic as the property is preserved under intersections and projections. Moreover, for all $B \in \C \setminus \{X\}$ we have that 
\begin{equation}
\label{eq:P}
\{A \in \C \setminus \{X \}: |A|_{\E}\le |B|_{\E}\} = \{A \in K(X): (A, B) \in P\}.
\end{equation}
Indeed, the containment $\subseteq$  of the above equality is clear. To see the containment $\supseteq$, let $(A,B) \in P$ and $x\in \lo$ be such that $(x,B) \in \cR$ and $(x,A) \in \cS$. Applying  Condition (R) to our set $B \in \C \setminus \{X\}$, we have that $x \in \wf$ and $|x|^* \le |B|_{\E}$. As $x \in \wf$ and $(x,A) \in \cS$, by Condition (S) we have that $A \in \C $ and $|x|^* = |A|_{\E}$. As $|x|^* >0$, we have that $A \neq X$. Hence, we have that $A \in \C \setminus \{X\}$ with $|A|_{\E} \le |B|_{\E}.$

Now, let $g\colon \lo \to \lo$ be a Borel function such that $x \in \wf$ if and only if $g(x) \in \wf$  and \[|g(x)|^* = |x|^* +1 \ \ \forall x\in \wf.\]  Indeed, this may be defined as follows: for $x \in \lo$ define $g(x)$ by $g(x)(2m,2n) = x(m,n)$, $g(x)(1,1)=1$, $g(x)(2m,1) =1$ for all $m \in D^*(x)$ and $g(x)(m,n)$ equals zero elsewhere.
%In order to obtain $Q$, fix a Borel function $x \mapsto x'$ from $\lo$ to $\lo$ such that $|x'|^* = |x|^* +1$ and $x \in \wf$ if and only if $x' \in \wf$.

Now, let
\[ Q = \{(A,B)  \in K(X) ^2 : \exists x \in \lo \textit{ such that } (g(x),B) \in \cR \ \& \ (x,A) \in \cS\}. \]
As $g$ is Borel, we see that $Q $ is analytic. Moreover, for all $B \in \C \setminus \{X\}$ we have that 

\begin{equation}
\label{eq:Q}    
\{A \in \C \setminus \{X \}: |A|_{\E} < |B|_{\E}\} = \{A \in K(X): (A, B) \in Q\}.
\end{equation} 
Indeed, the containment $\subseteq$  of the above equality is clear. To see the containment $\supseteq$, let $(A,B) \in Q$ and $x\in \lo$ be such that $(x',B) \in \cR$ and $(x,A) \in \cS$. Applying  Condition (R) to our set $B \in \C \setminus \{X\}$,  we have that $g(x) \in \wf$ and $|g(x)|^* \le |B|_{\E}$. As $g(x) \in \wf$, we have that $x \in \wf$. Now, as $x \in \wf$ and  $(x,A) \in \cS$, by Condition (S) we have that $A \in \C $ and $|x|^* = |A|_{\E}$. As $|x|^* >0$, we have that $A \neq X$. Hence, we have that $A \in \C \setminus \{X\}$ and \[|A|_{\E} = |x|^* < |g(x)|^* \le |B|_{\E},\] this concludes the argument. 

Use \eqref{eq:P}, \eqref{eq:Q}, and Proposition~\ref{altpi11} to conclude that the $\E$ rank on $\C \setminus \{X\}$ is a coanalytic rank.

 It remains to prove the existence of the analytic sets $\cR,\cS$ with the required properties.

We define
\begin{align*}
  \cR = \{(x,A) \in \lo \times K(X) : \exists h \in K(X)^{\N} \text{ s.t. } \\ 
  h(0) = A,  \\ 
  \forall m \in D^*(x), \ h(m) \neq X, \\
  \forall m \in D^*(x) \setminus \{0\},\text{ } \overline{\cup_{n \lex m} \E(h(n))} \subseteq h(m)  \} 
\end{align*}

\begin{align*}
    \cS = \{(x,A) \in \lo \times K(X) : \exists h \in K(X)^{\N} \text{ s.t. } \\ 
    h(0) = A, \\ 
    \forall m \in D^*(x), \ h(m) \neq X, \\
    \forall m \in D^*(x) \setminus \{0\},\text{ } \overline{\cup_{n \lex m} \E(h(n))} \subseteq h(m), \\
    \text{and } \overline{\cup_{m \in D^*(x)} \E(h(m))} = X \}.
\end{align*}

To show that $\cR$ and $\cS$ are analytic we define
\begin{align*}
  \cR' = \{(x,A,h) \in \lo \times K(X)\times K(X)^{\N} :  h(0) = A, \\ 
  \forall m \in D^*(x), \ h(m) \neq X, \\
  \forall m \in D^*(x) \setminus \{0\},\text{ } \overline{\cup_{n \lex m} \E(h(n))} \subseteq h(m)  \} 
\end{align*}

\begin{align*}
    \cS' = \{(x,A,h) \in \lo \times K(X)\times K(X)^{\N} :  
    h(0) = A, \\ 
    \forall m \in D^*(x), \ h(m) \neq X, \\
    \forall m \in D^*(x) \setminus \{0\},\text{ } \overline{\cup_{n \lex m} \E(h(n))} \subseteq h(m), \\
    \text{and } \overline{\cup_{m \in D^*(x)} \E(h(m))} = X \}.
\end{align*}

Note that $m \in D^*(x)$ if and only if $x(m,m)=1$.
Moreover,
\begin{equation}
\label{three}
\overline{\cup_{n \lex m} \E(h(n))} \subseteq h(m)
\end{equation}
if and only if  for all $ {n \lex m}$ we have that $E(h(n)) \subseteq h(m)$. 
With these observations, one can check that  $\cR'$ is Borel and hence $\cR$ is analytic as it is the projection of $\cR'$.
Furthermore, by Proposition~\ref{BorelLimit}, we have that \[\overline {\cup _n}: K(X)^{\N} \rightarrow K(X)\] defined by $\overline {\cup _n} (A_n) = \overline {\cup _n (A_n)}$ is Borel. This verifies that the last condition in the definition of $\cS'$ ($\overline{\cup_{m \in D^*(x)} \E(h(m))} = X $) is Borel. Putting all together, we have that 
$\cS'$ is Borel and $\cS$ and is analytic.  

Now, let us observe that $\cR$ and $\cS$ satisfy Conditions (R) and (S), respectively. Let $A \in \C \setminus \{X\}$. The containment $\supseteq$ in Condition (R) follows from the definition of $\cR$. Indeed, if $x \in \wf$ and $ |x|^* \le |A|_{\E}
$, for $m \in D^*(x)$, we let $h'(m) = \E^{\alpha_m}(A)$ where $\alpha_m$ is the order type of \[\{j \in D^*(x): j <^*_x m\}.\] Then we have that $h'(0)=\E ^0(A)=A$. For every $m\in D^*(x)$ we have that $\alpha_m<|A|_{\E}$ and hence $h'(m)\neq X$. Finally, for every $m\in D^*(x)\setminus \{0\}$ the containment $$\overline{\cup_{n \lex m} \E(h'(n))} \subseteq h'(m)$$ follows from \eqref{three}.
 %$$\overline{\cup_{n \lex m} \E^{\alpha_n+1}(A)} \subseteq \E^{\alpha_m}(A).$$
 This implies that $(x,A)\in \cR$.

Now we will prove the containment $\subseteq$ for Condition (R). Let $(x,A)\in \cR$, and $h$ the function given by the condition of $\cR$. If $m\in D^*(x)\setminus \{0\}$, then for some $\alpha < |A|_{\E}$ we have that \[\overline{\cup_{n \lex m} \E(h(n))} \nsupseteq \E^{\alpha+1} (A).\]
This is so, for otherwise, for all $\alpha < |A|_{\E}$ we would have that 
\[ 
\E^{\alpha+1} (A) \subseteq \overline{\cup_{n \lex m} \E(h(n))} \subseteq h(m) \neq X,
\]
implying that $\E^{\infty} (A)  = \overline{ \cup_{\alpha < |A|_{\E} }\E^{\alpha+1} (A) } \neq X$ and  that $A \notin \C$. Now we define $f\colon \lex \to |A|_{\E}$. Define $f(0) =0$ and for $m \in D^*(x) \setminus \{0\}$, let 
\[ f(m) = \text{ the least } \alpha < |A|_{\E} \text{ such that } \overline{\cup_{n \lex m} \E(h(n))} \nsupseteq \E^{\alpha+1} (A). 
\]
It suffices to show that $f$ is order preserving, i.e., $m  \lex p $ implies $f(m) < f(p)$. Indeed,  this means that $x \in \wf$ and $|x|^* \le |A|_{\E}$. 

We will now prove that $f$ is order-preserving. Let $m,p\in D^*(x)\setminus \{0\}$ with $m \lex p$. Then,
\[  \E^{f(m)}(A) = \overline{ \cup _{\alpha < f(m)} \E^{\alpha+1}(A) } \subseteq \overline{\cup_{n \lex m} \E(h(n))}.
\] 
Therefore, $\E^{f(m)} (A) \subseteq h(m)$, which implies that $\E^{f(m)+1} (A) \subseteq \E(h(m))$. As $m \lex p$, we have that 
\[
\E^{f(m)+1} (A) \subseteq \E(h(m)) \subseteq \overline{\cup_{q \lex p} \E(h(q))},
\]
concluding that $f(p) \ge f(m) +1$.

To see that $\cS$ satisfies Condition (S), let $x \in \wf$. That containment $\supseteq$ holds in Condition (S) follows from the definition of $\cS$ using the function $h'$ used for condition (R). We also use $|x|^*=|A|_{\E}$ to conclude that 
\[\overline{\cup_{m \in D^*(x)} \E(h'(m))} = X.
\]

The proof of containment $\subseteq$ also uses an equivalent order preserving function $f\colon \lex \to |A|_{\E}$. To prove that $|x|^*=|A|_{\E}$ we use that 
\[
 \overline{\cup_{m \in D^*(x)} \E(h(m))} = X.
\]

\end{proof}

\begin{corollary}
\label{cor:gamma}
    Let $X$ be a compact metrizable space. The $\Gamma$ rank on $C_{\Gamma}$ is a coanalytic rank. 
\end{corollary}
\begin{proof}
    Use Proposition~\ref{GammaBorel} and Theorem \ref{BorelExp}.
\end{proof}

Let $\cP$ be a dynamical pair assignment and $\cP _X \in \cP$.
The $\cP$ \textbf{rank} refers to the function $\varphi_{\cP}\colon C(X,X)\to \omega_1$ given by $\varphi_{\cP}(T)=|\cP_X(T)|_{\Gamma}$ (see Section \ref{sec:gamma}).

%We also define\[C_{\cP _X}=\{T \in C(X,X): \cP_X(T)\in C_{\Gamma}\}.
% |\cP_X(T)|_{\Gamma} < \omega _1\}.\]
For the definition of $R(\cP_X)$ see Definition \ref{def:R_p}.
\begin{theorem}\label{mainpi11}
Let $\cP$ be a dynamical pair assignment and $\cP _X \in \cP$. The $\cP$ rank, $\varphi_{\cP}$, on %$C_{\cP _X}$ 
$R(\cP_X)$ is a coanalytic rank.

%and ${\mathcal C} = \{T \in C(X,X):  |\cP_X(T)|_{\Gamma} < \omega _1\}$. Then, $\varphi: {\mathcal C} \rightarrow \omega_1$ defined by $\varphi (T) =  |\cP(T)|_{\Gamma}$ is a coanalytic rank on $C(X,X)$.
\end{theorem}
\begin{proof}
We already showed that $R(\cP_X)$ is a coanalytic subset (Proposition \ref{prop:R_p}). This can also be checked by noting that
\begin{equation}
\label{eq:1}
    \cP_X^{-1}(C_{\Gamma}) = R(\cP_X).
\end{equation}
Since $C_{\Gamma}$ is coanalytic (Theorem \ref{BorelExp}) and $\cP_X$ is Borel (by definition) we obtain that $R(\cP_X)$ is coanalytic.

Applying Corollary~\ref{cor:gamma} and the definition of coanalytic rank, there exist relations $P, Q \subseteq X^2$, with $P$ analytic and $Q$ coanalytic, such that for all $y \in C_{\Gamma}$ we have that 
\begin{align}
\label{eq:rank}
    \{x \in C_{\Gamma}: \varphi_{\Gamma} (x) \le \varphi_{\Gamma}(y)\} &= \{x \in X: (x,y) \in P\} \notag \\
    &= \{x \in X: (x,y) \in Q\}.
\end{align}

%By the definition of coanalytic rank, there are sets $P, Q \in K(X \times X ) ^2$, as in the Definition~\ref{pi11}, which verify that $ | \cdot |_{\Gamma}$ is a coanalytic rank on $\C '$. 
Let $P'=(\cP_X \times \cP_X)^{-1}(P)$ and $Q'=(\cP_X \times \cP_X)^{-1}(Q)$. Using these sets, one can see that $\varphi_{\cP}\colon R(\cP_X)\to \omega_1$ satisfies the definition of coanalytic rank.  Indeed, use the fact that $\cP_X$ is Borel (so $P'$ is analytic and $Q'$ coanalytic), \eqref{eq:1}, and \eqref{eq:rank}.
\end{proof}

The following result is known as the Boundedness Theorem for coanalytic ranks \cite[Theorem 35.23]{Kechris}. 
\begin{theorem}
\label{thm:bound}
Let $X$ be a Polish space, $C\subset X$ a coanalytic subset and $\varphi$ be a coanalytic rank on $C$. If $A \subseteq C$ is analytic, then C is Borel if and only if $\varphi$ is bounded on C.
\end{theorem}

We are now prepared to prove the characterization of when $R(\cP_X)$ is Borel.
\begin{proof}[Proof of Theorem~\ref{thm:GammaBorel}]
Combine Theorem \ref{thm:bound} and Theorem \ref{mainpi11}.
\end{proof}

\bibliographystyle{plain}
\bibliography{references}

\begin{thebibliography}{10}

\bibitem{auslander1988minimal}
Joseph Auslander.
\newblock {\em Minimal Flows and Their Extensions}.
\newblock Mathematics Studies. Elsevier, 1988.

\bibitem{barbieri2020}
Sebastián Barbieri and Felipe García-Ramos.
\newblock A hierarchy of topological systems with completely positive entropy.
\newblock {\em Journal d'Analyse Mathematique}, 143:639–680, 2021.

\bibitem{blanchard93}
Fran{\c{c}}ois Blanchard.
\newblock A disjointness theorem involving topological entropy.
\newblock {\em Bulletin de la Soci{\'e}t{\'e} Math{\'e}matique de France}, 121(4):465--478, 1993.

\bibitem{darjiGR}
Udayan~B. Darji and Felipe Garc{\'\i}a-Ramos.
\newblock Local entropy theory and descriptive complexity.
\newblock {\em arXiv:2107.09263}.

\bibitem{darji2021note}
Udayan~B Darji and Felipe Garc{\'\i}a-Ramos.
\newblock A note on derivatives, expansions and ${\Pi}^{1}_1$-ranks.
\newblock {\em arXiv:2107.09866}.

\bibitem{garciagutman}
Felipe Garc{\'\i}a-Ramos and Yonatan Gutman.
\newblock Local mean dimension theory for sofic group actions.
\newblock {\em Groups, Geometry and Dynamics, in press}.

\bibitem{garcia2024local}
Felipe Garc{\'\i}a-Ramos and Hanfeng Li.
\newblock Local entropy theory and applications.
\newblock {\em arXiv:2401.10012}.

\bibitem{glasner2009local}
Eli Glasner and Xiangdong Ye.
\newblock Local entropy theory.
\newblock {\em Ergodic Theory and Dynamical Systems}, 29(2):321, 2009.

\bibitem{huang2007relative}
Wen Huang, Xiangdong Ye, and Guohua Zhang.
\newblock Relative entropy tuples, relative u.p.e. and c.p.e. extensions.
\newblock {\em Israel Journal of Mathematics}, 158(1):249--283, 2007.

\bibitem{Kechris}
Alexander~S. Kechris.
\newblock {\em Classical Descriptive Set Theory}.
\newblock Graduate Texts in Mathematics. Springer-Verlag, New York, 1995.

\bibitem{KerrLiMA}
David Kerr and Hanfeng Li.
\newblock Independence in topological and {C}*-dynamics.
\newblock {\em Mathematische Annalen}, 338:869--926, 2007.

\bibitem{kerr2016ergodic}
David Kerr and Hanfeng Li.
\newblock {\em Ergodic Theory: Independence and Dichotomies}.
\newblock Monographs in Mathematics. Springer, 2016.

\bibitem{salo2019entropy}
Ville Salo.
\newblock Entropy pair realization.
\newblock {\em Ergodic Theory and Dynamical Systems}, 43(7):2471--2488.

\bibitem{veech1968equicontinuous}
William~A Veech.
\newblock The equicontinuous structure relation for minimal abelian transformation groups.
\newblock {\em American Journal of Mathematics}, 90(3):723--732, 1968.

\bibitem{westrick2019topological}
Linda Westrick.
\newblock Topological completely positive entropy is no simpler in $\mathbb{Z}^{2}$-{SFT}s.
\newblock {\em arXiv:1904.11444}.

\end{thebibliography}

\noindent Felipe Garc\'ia-Ramos,  fgramos@secihti.mx,
{\em  Physics Institute, Universidad Aut\'onoma de San Luis Potos\'i, and Faculty of Mathematics and Computer Science, Jagiellonian University.}\\

\noindent Udayan B. Darji, ubdarj01@gmail.com,
{\em  Department of Mathematics, University of Louisville.}\\

\end{document}